\documentclass[12pt]{amsart}
\usepackage[paperwidth=210mm,paperheight=297mm,inner=3cm,outer=3cm,top=2.5cm,bottom=2.6cm,marginparwidth=1.8cm]{geometry}

\usepackage{amssymb}
\usepackage{amsmath}
\usepackage{amsthm}
\usepackage{mathtools}
\usepackage{bm}
\usepackage{mathrsfs}

\usepackage[T1]{fontenc}
\usepackage{xcolor}

\usepackage{hyperref}

\usepackage[utf8]{inputenc}

\usepackage{graphicx}

\usepackage{cleveref}

\theoremstyle{plain}
\newtheorem{thm}{Theorem}
\numberwithin{thm}{section}

\newtheorem{prop}[thm]{Proposition}
\newtheorem{cor}[thm]{Corollary}

\theoremstyle{definition}
\newtheorem{ex}[thm]{Example}
\newtheorem{rem}[thm]{Remark}

\newtheorem{defi}[thm]{Definition}

\newcommand{\vol}{\text{vol}}
\newcommand{\conv}{\text{conv}}

\begin{document}

\bibliographystyle{abbrv}

\title[Comparing the sets of volume polynomials and Lorentzian Polynomials]{Comparing the sets of volume polynomials and Lorentzian Polynomials}

\author{Amelie Menges}
\email{Amelie.Menges@math.tu-dortmund.de}
\address{Fakultät für Mathematik, Technische Universität Dortmund, D-44227 Dortmund, Germany}

\begin{abstract}
Given $n$ convex bodies in the Euclidean space $\mathbb{R}^d$, we consider the set of homogeneous polynomials of degree $d$ in $n$ variables that can be represented as their volume polynomial. This set is a subset of the set of Lorentzian polynomials. Using our knowledge of operations that preserve the Lorentzian property, we give a complete classification of the cases when the two sets are equal.
\end{abstract}



\maketitle

\section*{Introduction}

For $n$ convex bodies $\mathcal{K}=(K_1,\ldots,K_n)$ in $\mathbb{R}^d$, Minkowski \cite{M} proved that the volume of their linear combination with non-negative coefficients is a homogeneous polynomial 
\[\vol(x_1K_1+\ldots+x_nK_n)=\sum_{\alpha\in\Delta_n^d}\frac{d!}{\alpha!}V_\alpha(\mathcal{K})x^\alpha\]
of degree $d$. This leads us to the problem whether a given homogeneous polynomial in $n$ variables of degree $d$ and with non-negative coefficients can be represented as the volume polynomial of $n$ convex bodies in the Euclidean space $\mathbb{R}^d$. 
Over the years, this question has been thoroughly studied. Most famously, Alexandrov and Fenchel independently from each other noticed that the coefficients satisfy the \emph{Alexandrov-Fenchel inequality}
\[V_\alpha(\mathcal{K})^2\geq V_{\alpha-e_i+e_j}(\mathcal{K})V_{\alpha-e_j+e_i}(\mathcal{K})\]
for every $\alpha\in\Delta_n^d$ with $\alpha_i, \alpha_j>0$ (see \cite{A} and \cite{F}). This inequality started a whole line of further inequalities that could be deduced using the Alexandrov-Fenchel inequality and that the coefficients of volume polynomials satisfy. These are often loosely referred to as the \emph{known inequalities}. In 1938, Heine \cite{H} managed to show that these inequalities describe the set of volume polynomials completely in the case $(n,d)=(3,2)$. He further proved that they are not enough to classify the set of volume polynomials in the case $(n,d)=(4,2)$. This was generalized by Shephard \cite{Sh} who constructed an example of a homogeneous polynomial in $d+2$ many variables for degree $d$ whose coefficients satisfy all known inequalities but which cannot be represented as the volume polynomial of any $d+2$ many convex bodies. Further, he proved that the known inequalities fully describe the set of volume polynomials in two variables of any degree.

Generalizing the known inequalities, Gurvits \cite{Gur} introduced the set of strongly log-concave polynomials and showed that it contains the set of volume polynomials. Furthermore, he conjectured that the sets are equal in the case of three variables. This was disproved by Huh and Brändén (see \cite{BH} and \cite{Hu}) who used the \emph{reverse Khovanskii-Teissier inequality} \cite[Thm.~5.7]{LX} 
\[\binom{d}{i}V_{(d-i)e_1+ie_2}(\mathcal{K})V_{ie_1+(d-i)e_3}(\mathcal{K})\geq V_{de_1}(\mathcal{K})V_{ie_2+(d-i)e_3}(\mathcal{K})\]
to construct an example of a strongly log-concave polynomial that cannot be a volume polynomial. They also introduced the set of Lorentzian polynomials which equals the set of strongly log-concave polynomials in the homogeneous case as they showed (see \cite[Thm.~2.30]{BH}).

Working with Lorentzian polynomials and particularly operations that preserve the Lorentzian property, Brändén und Huh answered a question of Gurvits and proved that the product of two Lorentzian polynomials is again Lorentzian (see \cite[Cor.~2.32]{BH}). On the other hand, polynomial factors of Lorentzian polynomials generally do not have to be Lorentzian. But as there are certain cases when we can deduce that the factors are Lorentzian polynomials, we can ask if the same is true for volume polynomials. Using our results for this problem, we can generalize the polynomials constructed by Shephard \cite{Sh} and Brändén and Huh \cite{BH}\cite{Hu} which are examples for Lorentzian polynomials that cannot be volume polynomials. Thus we can fully classify the cases in which the set of volume polynomials equals the set of Lorentzian polynomials.

Our main findings can be summarized as follows.

  \begin{itemize}
    \item[1.] If a Lorentzian polynomial can be factorized into polynomials with disjoint sets of variables, these factors are again Lorentzian. (Proposition \ref{21})
    \item[2.] If a volume polynomial can be factorized into polynomials with disjoint sets of variables, these factors are again volume polynomials. (Proposition \ref{24})
    \item[3.] If a Lorentzian polynomial can be written in the form $g=x_1^df$ with $\deg_{x_1}(f)=1$, the factors are Lorentzian. (Proposition \ref{22})
    \item[4.] If a volume polynomial can be written in the form $g=x_1^df$ with $\deg_{x_1}(f)=1$, the factors are volume polynomials. (Proposition \ref{25})
    \item[5.] The set of volume polynomials equals the set of Lorentzian polynomials if and only if $n\leq 2$, $d=1$ or $(n,d)=(3,2)$. (Theorem \ref{32})
  \end{itemize}

The paper is structured as follows. Section~\ref{sec:prelim} is
devoted to preliminaries; we recall basic definitions and properties of Lorentzian polynomials as well as volume polynomials.
Section~\ref{Sec:Factors} focuses on the factors of Lorentzian (resp. volume) polynomials and the question whether they are again Lorentzian (resp. volume) polynomials. Finally, in Section~\ref{Sec:Subset} we use our prior findings to fully classify when the set of volume polynomials equals the set of Lorentzian polynomials.

\bigskip
 \noindent \textbf{Acknowledgments.}
I would like to thank Ivan Soprunov as well as Khazhgali Kozhasov for important discussions concerning the case $n=3$ and about the geometric motivation behind the \emph{reverse Khovanskii-Teissier inequality}. I would also like to thank my advisor Daniel Plaumann for his guidance.

\section{Preliminaries}\label{sec:prelim}

By a \emph{convex body} $K$ in the Euclidean space $\mathbb{R}^d$, we mean a non-empty compact convex set. Especially, the convex bodies we consider do not need to have non-empty interior and thus can be less than full-dimensional. When we talk about the dimension of a convex body, we refer to the dimension of the smallest affine space containing the convex body.

We fix some notation and terminology concerning convex bodies and their volume polynomials. As a general reference, we suggest the monograph of Schneider \cite[Ch.~1]{Schn}. Let $n$ and $d$ be positive integers. We write $[n]\coloneqq\{1,\ldots,n\}$. The space of homogeneous polynomials of degree $d$ in $n$ variables over $\mathbb{R}$ is denoted by $H_n^d$.

Let $\mathcal{K}\coloneqq(K_1,\ldots,K_n)$ be convex bodies in $\mathbb{R}^d$. Their \emph{volume polynomial} is the homogeneous polynomial
\[\vol_{\mathcal{K}}(x)\coloneqq\vol(x_1K_1+\ldots+x_nK_n)\coloneqq\sum_{\alpha\in\Delta_n^d}\frac{d!}{\alpha!}V_\alpha(\mathcal{K})x^\alpha\]
for non-negative $x_1,\ldots,x_n$ and $x\coloneqq(x_1,\ldots,x_n)$. For a multi-index $\alpha\in\Delta_n^d$ we write
\[\alpha!\coloneqq\alpha_1!\cdot\ldots\cdot\alpha_n! \hspace{20px} \text{and} \hspace{20px} x^\alpha\coloneqq x_1^{\alpha_1}\cdot\ldots\cdot x_n^{\alpha_n}.\]
We generally assume $\mathcal{K}$ to be \emph{full-dimensional}, meaning that the affine dimension of the convex body
\[\sum\mathcal{K}\coloneqq\sum_{i=1}^nK_i\]
equals $d$. This guarantees that the volume polynomial is non-zero. We further use the notation 
\[V_\alpha(\mathcal{K})=V(\mathcal{K}^\alpha)=V(K_1^{\alpha_1},\ldots,K_n^{\alpha_n})\]
to refer to the mixed volume of the convex bodies $\mathcal{K}=(K_1,\ldots,K_n)$. The set of all volume polynomials is denoted by $V_n^d$. 

The mixed volumes satisfy several useful properties of which we will only list a few here. For a more thorough understanding, we refer to the monograph of Schneider \cite[Ch.~5]{Schn}.

\begin{prop}\label{11} 
	\begin{itemize}
		\item[(a)] The mixed volumes are non-negative and symmetric as well as multi-linear in the convex bodies (see Schneider \cite[Ch.~5.1]{Schn}).
		\item[(b)] For $a_1,\ldots,a_n\in\mathbb{R}^d$ and a regular matrix $A\in\mathbb{R}^{d\times d}$, let $T_{a_i,A}$ denote the affine transformation $\mathbb{R}^d\rightarrow\mathbb{R}^d, x\mapsto Ax+a_i$. We have \[\vol(x_1T_{a_1,A}(K_1)+\ldots+x_nT_{a_n,A}(K_n))=|\det(A)| vol(x_1K_1+\ldots+x_nK_n).\] Thus we may always assume that the considered convex bodies contain the origin (see Shephard \cite[p.~126]{Sh}).
		\item[(c)] For $i\in[n]$, the volume polynomial and mixed volumes carry the information $\deg_i(\vol_{\mathcal{K}})=\dim(K_i)$ and $\vol(K_i)=V_{de_i}(\mathcal{K})$ (see Gurvits \cite[Fact~A.7]{Gur2}).
	\end{itemize}
\end{prop}

For a $k$-dimensional subspace $E\subset\mathbb{R}^d$, we denote by $\vol_E$ (resp. $V_E$) the volume (resp. the mixed volume) in the space $E$. We omit the subspace in the notation if it can be deduced from the context. By $K|E$, we denote the orthogonal projection of a convex body $K\subset\mathbb{R}^d$ onto the space $E$. If some of the convex bodies we are considering lie in a common subspace, we can use this to represent the mixed volume in $\mathbb{R}^d$ as a product of the mixed volumes in the smaller subspace and its orthogonal complement.

\begin{prop}\label{12} \cite[Thm.~5.3.1]{Schn}
	Let $E$ be a $k$-dimensional subspace of $\mathbb{R}^d$ and let $L_1,\ldots,L_k\subset E$ as well as $K_1,\ldots,K_{d-k}\subset\mathbb{R}^d$ be convex bodies. We have
	\[\binom{d}{k}V(L_1,\ldots,L_k,K_1,\ldots,K_{d-k})=V_E(L_1,\ldots,L_k)V_{E^\perp}(K_1| E^\perp,\ldots,K_{d-k}| E^\perp),\]
	where $E^\perp$ refers to the orthogonal space of $E$. 
\end{prop}

The mixed volumes satisfy several useful inequalities, the most famous being the \emph{Alexandrov-Fenchel inequality} (see \cite{A} and \cite{F})
\[V(K_1,\ldots,K_n)^2\geq V(K_1^2,K_3,\ldots,K_n)V(K_2^2,K_3,\ldots,K_n).\]
A generalization of polynomials with coefficients satisfying this inequality leads us to the set of Lorentzian polynomials.

\begin{defi} \label{13} (see \cite{BH})
	A subset $J\subseteq\mathbb{N}^n$ is called \emph{M-convex} if for any $\alpha, \beta\in J$ and any index $i\in[n]$ with $\alpha_i>\beta_i$, there exists an index $j\in J$ with $\alpha_j<\beta_j$ and $\alpha-e_i+e_j,~\beta-e_j+e_i\in J$. We denote by $M_n^d$ the set of all polynomials in $H_n^d$ with non-negative coefficients and M-convex support. We further define the set of \emph{Lorentzian polynomials} as $L_n^1\coloneqq M_n^1$ and for $d\geq 2$ as
	\[L_n^d\coloneqq\{f\in M_n^d~\mid~\text{for all }\alpha\in\Delta_n^{d-2}:~\mathcal{H}_{\partial^\alpha f}~\text{has at most one positive eigenvalue}\},\]
	where $\mathcal{H}_f$ refers to the Hessian of a polynomial $f\in H_n^d$.
\end{defi}

The conditions for the Hessian matrices lead to the fact that the coefficients of Lorentzian polynomials always satisfy the \emph{Alexandrov-Fenchel inequality}. In fact, we have the inclusion $V_n^d\subseteq L_n^d$.

\begin{thm}\label{14} \cite[Thm.~4.1]{BH}
	Every volume polynomial is Lorentzian.
\end{thm}

As Lorentzian polynomials have been thoroughly studied, especially concerning operations that preserve their properties, we will focus on some basic notions here and refer the reader to the work of Brändén and Huh \cite{BH} for a broader understanding.

\begin{prop}\label{15}
	\begin{itemize}
		\item[(1)] The product of Lorentzian polynomials is Lorentzian (see \cite[Cor.~2.32]{BH}).
		\item[(2)] Let $A\in\mathbb{R}^{n\times m}_{\geq0}$ be a $(n\times m)$-matrix with non-negative entries. For a Lorentzian polynomial $f\in L_n^d$ and $x\coloneqq(x_1,\ldots,x_m)^\top$, we have $f(Ax)\in L_m^d$ (see  \cite[Thm.~2.10]{BH}).
	\end{itemize}
\end{prop}

Both of these properties can be transferred to volume polynomials.

\begin{rem}\label{16}
	\begin{itemize}
		\item[(1)] The product of volume polynomials is a volume polynomial.
		\item[(2)] Let $A\in\mathbb{R}^{n\times m}_{\geq0}$ be a $(n\times m)$-matrix with non-negative entries. For a volume polynomial $f\in V_n^d$ and $x\coloneqq(x_1,\ldots,x_m)^\top$, we have $f(Ax)\in V_m^d$ (see \cite[Ex.~1.2]{Gur})
	\end{itemize}
\end{rem}

Generally, the set of Lorentzian polynomials allows more operations that preserve it than the set of volume polynomials. 

\begin{prop}\label{17}\cite[Lem.~4.4]{BLP}
	Let $f\in L_n^d$ be a Lorentzian polynomial and let us write
	\[f(x_1,\ldots,x_n)=\sum_{i=0}^dx_n^{d-i}f_i(x_1,\ldots,x_{n-1}).\]
	Then $f_i$ is a Lorentzian polynomial of degree $i$ for every $i\in[d]$.
\end{prop}

In contrast to Proposition \ref{15} and Remark \ref{16}, an easy example shows that Proposition \ref{17} is not necessarily true for volume polynomials. Let us consider the polynomial
\begin{align}
	f\coloneqq~&x_5^3+x_5^2(x_1+x_2+x_3+\textstyle \frac{3}{2}x_4)+x_5(x_1x_2+x_1x_3+x_1x_4+x_2x_3+x_2x_4+x_3x_4)\notag \\
	&+x_1x_2x_3+\textstyle \frac{1}{2}x_1x_2x_4+\frac{1}{2}x_1x_3x_4+\frac{1}{2}x_2x_3x_4, \notag
\end{align}
which is the volume polynomial of the convex bodies
\begin{align}
	K_1&\coloneqq\conv(0,e_1),\notag \\
	K_2&\coloneqq\conv(0,e_2),\notag \\
	K_3&\coloneqq\conv(0,e_3),\notag \\
	K_4&\coloneqq\conv(0,\frac{1}{2}(e_1+e_2+e_3)), \notag\\
	K_5&\coloneqq\conv(0,e_1,e_2,e_3,e_1+e_2,e_1+e_3,e_2+e_3,e_1+e_2+e_3). \notag
\end{align}
According to \cite[p.~119]{H}, the elementary symmetric polynomial in four variables of degree two, and thus $f_2$, cannot be represented as a volume polynomial. In Section~\ref{Sec:Subset}, we will go into more depth as to why the polynomial $f_2$ cannot be represented as the volume polynomial of any four convex bodies in $\mathbb{R}^2$.

\section{Operations} 
\label{Sec:Factors}

We have noted that the product of Lorentzian (resp. volume) polynomials is again a Lorentzian (resp. volume) polynomial. Generally, the factors of Lorentzian or volume polynomials do not have to be either. For example, the polynomial \[f\coloneqq x^3+3x^2y+3xy^2=x(x^2+3xy+3y^2)\] is Lorentzian. Because of the fact that $f$ is bivariate, it is due to Shephard (\cite[Thm.~4]{Sh}) that $f$ is a volume polynomial. As can be easily verified though, the factor $x^2+3xy+3y^2$ is not Lorentzian and thus cannot be a volume polynomial (Theorem \ref{14}).

Nevertheless, there are certain cases where the factors of Lorentzian (resp. volume) polynomials are Lorentzian (resp. volume) polynomials.

\begin{prop}\label{21}
	Let $f\coloneqq gh\in L_{n_1+n_2}^{d_1+d_2}$ be a Lorentzian polynomial with factors $g\in H_{n_1}^{d_1}$ and $h\in H_{n_2}^{d_2}$ with non-negative coefficients and in distinct variables $x_1,\ldots,x_{n_1}$ and $y_1,\ldots,y_{n_2}$. Then both factors are again Lorentzian.
\end{prop}

\begin{proof}
	First, we look at the case that the polynomial $f$ is of the form $x_{n+1}^{d_1}g\in L_{n+1}^{d_1+d_2}$ for a polynomial $g\in H_n^{d_2}$ in the variables $x_1,\ldots,x_n$. As the support of $f$ is M-convex, the same is true for the support of $g$. Let $\alpha\in\Delta_n^{d_2-2}$ and take a look at the Hessian of $\partial^\alpha g$ which is precisely the Hessian of $\frac{1}{d_1!}\partial^{(\alpha,d_1)}f=\frac{1}{d_1!}\partial^{\alpha}\partial_{n+1}^{d_1}f$. From this, we immediately conclude that the Hessian of $\partial^\alpha g$ has at most one positive eigenvalue and that $g$ is in $L_n^{d_2}$.
	
	For the general case, let $f\coloneqq gh\in L_{n_1+n_2}^{d_1+d_2}$ be a Lorentzian polynomial with $g$ and $h$ having distinct variables. We define the $(n_1+n_2)\times(n_2+1)$-matrix
	\[A\coloneqq\begin{pmatrix}
		\frac{1}{|g|_1} & 0 &\ldots & \ldots & 0 \\
		\vdots & \vdots & \ddots & \ddots & \vdots \\
		\frac{1}{|g|_1} & 0 &\ldots & \ldots & 0 \\
		0 & 1 & 0 & \ldots & 0 \\
		\vdots & 0 & \ddots & \ddots & \vdots \\
		\vdots & \vdots & \ddots & 1 & 0 \\
		0 & 0& \ldots & 0 & 1 
	\end{pmatrix},\]
	where $|g|_1$ defines the sum of the (non-negative) coefficients of $g$. Taking this matrix we know that the polynomial
	\[f\left(A\begin{pmatrix}
		y \\ x_1 \\ \vdots \\ x_{n_2}   
	\end{pmatrix}\right)=y^{d_1}h\]
	is Lorentzian by Proposition \ref{15}. According to the case noted above, $h$ is a Lorentzian polynomial and analogously, the same is true for $g$.
\end{proof}

The next proposition shows that we can skip the restriction of distinct variables if we restrict the degree of the (only) common variable of the two factors.

\begin{prop}\label{22}
	Let $f\coloneqq x_1^{d_1}g\in L_n^{d_1+d_2}$ be a Lorentzian polynomial with a polynomial $g\in H_n^{d_2}$ such that $\deg_1(g)\leq1$. Then the polynomial $g$ is also Lorentzian.
\end{prop}

\begin{proof}
	The support of $g$ is clearly M-convex. We take an $\alpha\in\Delta_n^{d_2-2}$ and get 
	\[\partial^{\alpha+d_1e_1}f=\partial^\alpha d_1!(1+d_1x_1\partial_1)g.\] The Hessian of $\partial^{\alpha+d_1e_1}f$ has at most one positive eigenvalue, so the same goes for the Hessian of $\partial^\alpha(1+d_1x_1\partial_1)g$. As $\deg_1(g)\leq 1$, we may assume $\alpha_1\in\{0,1\}$ without loss of generality. In the case of $\alpha_1=1$, we have
	\[\partial^\alpha (1+d_1x_1\partial_1)g=(1+d_1)\partial^\alpha g\] and we know that the Hessian of $\partial^\alpha g$ has at most one positive eigenvalue. If $\alpha_1=0$, the Hessian matrices $A$ of $\partial^\alpha(1+d_1x_1\partial_1)g$ and $B$ of $\partial^\alpha g$ satisfy the relation 
	\[A=\text{diag}(d_1+1,1\ldots,1)\cdot B\cdot\text{diag}(d_1+1,1,\ldots,1).\]
	As this does not change the number of positive eigenvalues, $g$ must be Lorentzian.
\end{proof}

Using the same technique as in the proof of Proposition \ref{21}, we can deduce the following corollary.

\begin{cor}\label{23}
	Let $f\coloneqq gh\in L_{n_1+n_2-1}^{d_1+d_2}$ be a Lorentzian polynomial, such that $g\in H_{n_1}^{d_1}$ and $h\in H_{n_2}^{d_2}$ share (exactly) one variable $x_1$ and the polynomial $h$ has at most degree $1$ in $x_1$. Then $h$ is Lorentzian.
\end{cor}

\begin{rem}
	Originally, we looked at the special case of multiaffine factors and proved that for a Lorentzian polynomial $f\coloneqq gh\in L_{n_1+n_2-1}^{d_1+d_2}$ with multiaffine factors $g\in H_{n_1}^{d_1}$ and $h\in H_{n_2}^{d_2}$ sharing one variable, both factors are again Lorentzian. Motivated by this, it was possible to generalize the techniques we used and thus come to the proof of Proposition \ref{22} and Corollary \ref{23}.
\end{rem}

In order to draw the connection to volume polynomials, we need to use different techniques. Basically, we can use the geometric aspects of the given convex bodies to transfer the results to volume polynomials.

\begin{prop}\label{24}
	Let $f\coloneqq gh\in V_{n_1+n_2}^{d_1+d_2}$ be a volume polynomial with factors $g\in H_{n_1}^{d_1}$ and $h\in H_{n_2}^{d_2}$ with non-negative coefficients and in distinct variables $x_1,\ldots,x_{n_1}$ and $y_1,\ldots,y_{n_2}$. Then both factors are again volume polynomials.
\end{prop}

\begin{proof}
	We proceed as we did in the case of Lorentzian polynomials and first assume that $f\coloneqq x_{n+1}^{d_1}g\in V^{d_1+d_2}_{n+1}$ is a volume polynomial with $g\in H_n^{d_2}$ being a polynomial in the variables $x_1,\ldots,x_n$. Let $\mathcal{K}\coloneqq (K_1,\ldots,K_{n+1})$ be the convex bodies in $\mathbb{R}^{d_1+d_2}$ associated to the corresponding variables $x_1,\ldots,x_{n+1}$. As we have $\deg_{n+1}f=d_1$, the convex body $K_{n+1}$ must lie in a $d_1$-dimensional subspace $E\subseteq\mathbb{R}^{d_1+d_2}$. Because of
	\[V(K_1^{\alpha_1},\ldots,K_n^{\alpha_n},K_{n+1}^{d_1-1})=0\]
	for all $\alpha\in\Delta_n^{d_2+1}$, we can assume $K_1,\ldots,K_n\subseteq E^\perp$.
	Now we take $\alpha\in\Delta_n^{d_2}$ and we have
	\begin{align}
		V_{(\alpha,d_1)}(K)&=\binom{d_1+d_2}{d_1}^{-1}V(K_{n+1}^{d_1})V(K_1^{\alpha_1},\ldots,K_n^{\alpha_n}) \notag \\
		&=\frac{d_1!d_2!}{(d_1+d_2)!}V(K_{n+1}^{d_1})V(K_1^{\alpha_1},\ldots,K_n^{\alpha_n}) \notag
	\end{align}
	by Proposition \ref{12}. For the volume polynomial $f$, this leads to 
	\begin{align}
		f&=\sum_{\alpha\in\Delta_n^{d_2}}\frac{(d_1+d_2)!}{\alpha!d_1!}V_{(\alpha,d_1)}(\mathcal{K})x^\alpha x_{n+1}^{d_1} \notag \\
		&=x_{n+1}^{d_1}\sum_{\alpha\in\Delta_n^{d_2}}\frac{d_2!}{\alpha!} \vol_{d_1}(K_{n+1})V(K_1^{\alpha_1},\ldots,K_n^{\alpha_n})x^\alpha \notag \\
		&=x_{n+1}^{d_1} \vol_{d_1}(K_{n+1})\vol(x_1K_1+\ldots+x_nK_n), \notag 
	\end{align}
	where $\vol_{d_1}$ refers to the $d_1$-dimensional volume in the subspace $E$ of $\mathbb{R}^{d_1+d_2}$. Thus $g$ is a volume polynomial. \\
	For the general case, we use a transformation of the variables as before and obtain the result.
\end{proof}

\begin{prop}\label{25}
	Let $f\coloneqq x_1^{d_1}g\in V_n^{d_1+d_2}$ be a volume polynomial with a polynomial $g\in H_n^{d_2}$ such that $\deg_1(g)\leq1$. Then the polynomial $g$ is also a volume polynomial.
\end{prop}

\begin{proof}
	Let $f$ be the volume polynomial of the convex bodies $\mathcal{K}\coloneqq(K_1,\ldots,K_n)$ in $\mathbb{R}^{d_1+d_2}$. Without loss of generality we can assume $\dim(K_1)=d_1+1$ and $K_2,\ldots,K_n\subseteq V$ for a $d_2$-dimensional subspace $V\subseteq\mathbb{R}^{d_1+d_2}$. We denote by $U_1$ the $(d_1+1)$-dimensional subspace containing $K_1$ and get $U_1\cap V=\mathbb{R}v$ for a vector $v\in\mathbb{R}^{d_1+d_2}$. We can now write $U_1=U+\mathbb{R}v$ for a $d_1$-dimensional subspace $U\subseteq\mathbb{R}^{d_1+d_2}$ and without loss of generality we assume $U=V^\perp$, particularly $v\in U^\perp$. We write $C_1\coloneqq K_1|U$ and chose the length of $v$ such that we get
	\[\vol_{d_1+1}(K_1)=\vol_{d_1+1}(C_1+\conv(v))=\|v\|\vol_{d_1}(C_1).\]
	For an $\alpha\in\Delta_n^{d_2}$ with $\alpha_1=0$, we have
	\[\frac{(d_1+d_2)!}{d_1!\alpha!}V(K_1^{d_1},K_2^{\alpha_2},\ldots,K_n^{\alpha_n})=\frac{d_2!}{\alpha!}V_U(C_1)V_V(K_2^{\alpha_2},\ldots,K_n^{\alpha_n})
	\]
	and for an $\alpha\in\Delta_n^{d_2}$ with $\alpha_1=1$, we get
	\begin{align}
	&\frac{(d_1+d_2)!}{(d_1+1)!\hat{\alpha}}V(K_1^{d_1+1},K_2^{\alpha_2},\ldots,K_n^{\alpha_n}) \notag \\
	=&\frac{(d_1+d_2)!}{(d_1+1)!\hat{\alpha}}V((C_1+\conv(v))^{d_1+1},K_2^{\alpha_2},\ldots,K_n^{\alpha_n}) \notag \\
	=&\frac{(d_1+d_2)!}{(d_1+1)!\hat{\alpha}}\sum_{i=0}^{d_1+1}\binom{d_1+1}{i}V(C_1^{d_1+1-i},\conv(v)^i,K_2^{\alpha_2},\ldots,K_n^{\alpha_n}) \notag \\
	=&\frac{d_2!}{\hat{\alpha}!}\binom{d_1+d_2}{d_1}V(C_1^{d_1},\conv(v),K_2^{\alpha_2},\ldots,K_n^{\alpha_n}) \notag \\
	= &\frac{d_2!}{\alpha!}V_U(C_1)V_V(\conv(v),K_2^{\alpha_2},\ldots,K_n^{\alpha_n})\notag,
    \end{align}
	where $\hat{\alpha}$ refers to $(\alpha_2,\ldots,\alpha_n)$. Hence, we have
	\[f=x_1^{d_1}\vol_{d_1}(C_1)\vol(x_1\conv(v)+x_2K_2+\ldots+x_nK_n)\]
	and thus $g$ is a volume polynomial.
\end{proof}

\begin{cor}\label{26}
	Let $f\coloneqq gh\in V_{n_1+n_2-1}^{d_1+d_2}$ be a volume polynomial such that $g\in H_{n_1}^{d_1}$ and $h\in H_{n_2}^{d_2}$ only share one variable $x_1$ and the polynomial $h$ has at most degree $1$ in $x_1$. Then $h$ is a volume polynomial.
\end{cor}

\begin{rem}
	Similarly to the corresponding results for Lorentzian polynomials, Proposition \ref{25} and Corollary \ref{26} were also motivated by the special case of multiaffine factors. Especially for volume polynomials, multiaffine polynomials allow an explicit description of the corresponding convex bodies due to Proposition \ref{11}. This allows a straight forward approach for the proof of Proposition \ref{25}, which can be generalized as seen above. 
\end{rem}

The above results illustrate how we can use our knowledge of Lorentzian polynomials to obtain new information on volume polynomials. But as mentioned before, the vast majority of results and operations for Lorentzian polynomials are not transferable to volume polynomials. Instead, we often need further restrictions or some adjusting of the results to be able to transfer the operations preserving the Lorentzian property to volume polynomials. We have seen one such example of a non-transferable result in Proposition \ref{17}. As we explicitly do not require the convex bodies to have non-empty interior, we can transfer at least parts of Proposition \ref{17} to volume polynomials.

\begin{prop}\label{27}
	Let $f\in V_n^d$ be the volume polynomial of $n$ convex bodies $\mathcal{K}\coloneqq(K_1,\ldots,K_n)$ in $\mathbb{R}^d$ and let us write
	\[f(x_1,\ldots,x_n)=\sum_{i=0}^dx_n^{d-i}f_i(x_1,\ldots,x_{n-1}).\]
	Then $f_d$ is a volume polynomial of degree $d$ and $f_{d-m}$ for $m\coloneqq\dim(K_n)$ is a volume polynomial of degree $d-m$.
\end{prop}

\begin{proof}
	We have $f_d=f(x_1,\ldots,x_{n-1},0)=\vol(x_1K_1+\ldots+x_{n-1}K_{n-1})$. Let $U\subseteq\mathbb{R}^d$ be the $m$-dimensional subspace with $K_n\subseteq U$. We have
	\begin{align}
		f_{d-m}&=\sum_{\alpha\in\Delta_{n-1}^{d-m}}\frac{d!}{m!\alpha!}V(\mathcal{K}^\alpha,K_n^m)x^\alpha \notag \\
		&=\sum_{\alpha\in\Delta_{n-1}^{d-m}}\frac{(d-m)!}{\alpha!}V_U(K_n^m)V_{U^\perp}((K|U^\perp)^\alpha)x^\alpha \notag \\
		&=\vol_m(K_1)\vol(x_1(K_1| U^\perp)+\ldots +x_{n-1}(K_{n-1}| U^\perp)). \notag \qedhere
	\end{align}
\end{proof}

\section{Volume Polnomials as a subset of Lorentzian Polynomials}
\label{Sec:Subset}

The \emph{Alexandrov-Fenchel inequality} (see \cite{A} and \cite{F}), being the first major restriction for sequences that can be realized as a sequence of coefficients of a volume polynomial, started a long line of further inequalities that can be deduced from it. The set of homogeneous polynomials with coefficients satisfying these inequalities contains the set of Lorentzian polynomials (\cite[Ex.~1.2(3)]{Gur} and \cite[Prop.~4.4]{BH}) which allows us to solely focus on this smaller set as Brändén and Huh found that every volume polynomial is Lorentzian (\cite[Thm.~4.1]{BH}).

We denote by $AF_n^d$ the set of homogeneous polynomials in $n$ variables of degree $d$ with non-negative coefficients satisfying the \emph{Alexandrov-Fenchel inequality} as well as the resulting inequalities \cite[p.~132]{Sh}
\[V(K^\alpha,K_i^{r-1},K_j)V(K^\alpha,K_i,K_j^{r-1})\geq V(K^\alpha,K_i^r)V(K^\alpha,K_j^r)\] for $\alpha\in\Delta_n^{d-r}$ and \[(-1)^r\det(V(K^\beta,K_i,K_j)_{i,j\in[r]})\leq0\] for $\beta\in\Delta_n^{d-2}$ and $r\leq n$. Considering the polynomial
\[g\coloneqq c_{111}x_1^3+3c_{223}x_2^2x_3+3c_{233}x_2x_3^2\] 
with $c_{111},c_{223},c_{233}>0$ which lies in $AF_3^3$ but not in $L_3^3$, one sees that focusing on the set $L_n^d$ instead of $AF_n^d$ already reduces the number of polynomials. Thus going forward, we solely regard the set $V_n^d$ as a subset of $L_n^d$.

Shephard proved \cite[Thm.~4]{Sh} that for any degree $d\in\mathbb{N}$, we have
\[V_2^d=L_2^d\]
and he further proved \cite[Thm.~5]{Sh} that for $(d+2)$-many variables, the inclusion \[V_{d+2}^d\subsetneq L_{d+2}^d\]
is strict. This generalized a result of Heine \cite[p.~119]{H} for polynomials in four variables and of degree two. To illustrate the idea behind the proof, we will mention Heine's example here.

\begin{ex}\label{31}\cite[p.~119]{H}
	The elementary symmetric polynomial in four variables of degree two \[f\coloneqq x_1x_2+x_1x_3+x_1x_4+x_2x_3+x_2x_4+x_3x_4\] is Lorentzian as can be seen straight forwardly by computing the Hessian matrix. If it were the volume polynomial of convex bodies $K_1,\ldots,K_4\subseteq\mathbb{R}^2$, these would have to be line-segments by Remark \ref{11}. Without loss of generality, we assume $K_i=\conv(0,e_i)$ for $i=1,2$ and $K_3=\conv(0,a)$, $K_4=\conv(0,b)$ for $a,b\in\mathbb{R}^2$. Computing the mixed volumes of these convex bodies leads to
	\[	1=\pm a_i = \pm b_i=\pm(a_1b_2-a_2b_1)\]
	for $i=1,2$ and thus to a contradiction.
\end{ex}

This example also illustrates why it is often useful to first refer to multiaffine polynomials as they allow an easy computation of the mixed volumes, which would otherwise be more difficult (see \cite{BF} and \cite{DGH}).

In the case of three variables, Heine \cite[p.~118]{H} proved
\[V_3^2=L_3^2.\] 
Later, Gurvits \cite[Conj.~5.1]{Gur} conjectured that this might be true for all degrees. This was disproved by Brändén and Huh (\cite[Fn.~15]{BH} and \cite[Ex.~14]{Hu}), who constructed the Lorentzian polynomial
\[f\coloneqq 14x_1^3+6x_1^2x_2+24x_1^2x_3+12x_1x_2x_3+6x_1x_3^2+3x_2x_3^2,\]
which cannot be a volume polynomial as the coefficients do not satisfy the \emph{reverse Khovanskii-Teissier inequality} (\cite[Thm.~5.7]{LX}). This inequality states that for three convex bodies $K_1,K_2,K_3$ in $\mathbb{R}^d$, the mixed volumes satisfy
\[\binom{d}{k}V(K_1^{d-k},K_2^k)V(K_1^k,K_3^{d-k})\geq V(K_1^d)V(K_2^k,K_3^{d-k})\]
for all non-negative integers $k\leq d$.

To give some geometric motivation for the inequality, we assume that we have three convex bodies $K_1,K_2,K_3\subseteq\mathbb{R}^d$ with $\dim(K_2)=k\leq d$. Let $U\subseteq\mathbb{R}^d$ be the $k$-dimensional subspace with $K_2\subseteq U$. Proposition \ref{12} leads us to
\begin{align}
	\binom{d}{k}V(K_2^k,K_3^{d-k})&=\vol_k(K_2)\vol_{d-k}(K_3| U^\perp), \notag\\\
	\binom{d}{k}V(K_1^{d-k},K_2^k)&=\vol_{d-k}(K_1| U^\perp)\vol_k(K_2), \notag \\
	\binom{d}{k}V(K_1^k,K_3^{d-k})&\geq \binom{d}{k}V(K_1^{k},(K_3| U^\perp)^{d-k})=\vol_k(K_1| U)\vol_{d-k}(K_3| U^\perp). \notag
\end{align}
By approximating the volume of $K_1$, we get
\begin{align}
	V(K_1^d)&\leq \vol_k(K_1| U)\vol_{d-k}(K_1|U^\perp) \notag \\
	&\leq \binom{d}{k}\frac{V(K_1^k,K_3^{d-k})V(K_1^{d-k},K_2^k)}{V(K_2^k,K_3^{d-k})}.\notag
\end{align}
As was communicated by Ivan Soprunov, this can be used to show that the above example of a polynomial in $L_3^3\setminus V_3^3$ by Brändén and Huh cannot be a volume polynomial (without using Hodge theory, as in their proof). In the general case, when the convex bodies $K_1,K_2,K_3\subseteq\mathbb{R}^d$ have dimension greater than $k$ or $d-k$, one cannot use the above technique to see that the mixed volumes satisfy the \emph{reverse Khovanskii-Teissier inequality}.

Using the above polynomials and our prior results, we are now in the position to prove our main theorem and thus to fully classify when the inclusion $V_n^d\subseteq L_n^d$ is strict. First, the case $n=1$ obviously leads to $V_1^d\subseteq L_1^d$ for all $d \in \mathbb{N}$. Second, the case $d=1$ obviously leads to $V_n^1\subseteq L_n^1$ for all $n \in \mathbb{N}$. The remaining cases are solved in the following.
 
\begin{thm}\label{32}
	Let $d,n \geq 2$. The sets $V_n^d$ and $L_n^d$ coincide if and only if $n=2$ or $(d,n)=(2,3)$. 
\end{thm}

\begin{proof}
	Shephard \cite[Thm.~4]{Sh} proved that the sets are equal for $n=2$ and Heine \cite[p.~118]{H} proved the same for $(d,n)=(2,3)$. We define the polynomial
	\[f_k\coloneqq x_3^k(14x_1^3+6x_1^2x_2+24x_1^2x_3+12x_1x_2x_3+6x_1x_3^2+3x_2x_3^2)\]
	for $k\in\mathbb{N}_0$ which is a Lorentzian polynomial in $L_3^{3+k}$ by Proposition \ref{15} as it is the product of two Lorentzian polynomials. By the results of Brändén and Huh (\cite[Fn.~15]{BH} and \cite[Ex.~14]{Hu}), the second factor cannot be realized as a volume polynomial. By Proposition \ref{25}, the polynomial $f_k$ cannot lie in $V_3^{3+k}$ either. Hence, we have $V_3^{3+k}\subsetneq L_3^{3+k}$ for all $k\in\mathbb{N}_0$. The polynomial 
	\[f\coloneqq x_1x_2+x_1x_3+x_1x_4+x_2x_3+x_2x_4+x_3x_4\]
	 leads to the strict inclusion $V_4^2\subsetneq L_4^2$ (Shephard \cite[Thm.~5]{Sh} and Heine \cite[p.~119]{H}). Given $n$ with $V_n^d\subsetneq L_n^d$, we can deduce $V_{n+1}^d\subsetneq L_{n+1}^d$ by taking a polynomial $g\in L_n^d\setminus V_n^d$. By Proposition \ref{15}, the polynomial
	\[g(x_1,\ldots,x_n+x_{n+1})\]
	is a Lorentzian polynomial in $L_{n+1}^d$. If the new polynomial was a volume polynomial, the same would be true for $g$ by Remark \ref{16}.
\end{proof}

\vspace{25px}

\begin{thebibliography}{10}
	
  \bibitem{A}
  A.D.~Alexandrov.
  \newblock {\em Selected works. {P}art {I}}.
  \newblock Gordon and Breach Publishers, Amsterdam, 1996.
  
  \bibitem{BF}
  I.~Bárány,~Z.~Füredi.
  \newblock{\em Computing the volume is difficult.}
  \newblock{\em Discrete Comput. Geom.}, 2(4): 319--326, 1987.
    
  \bibitem{BH}
  P.~Brändén,~J.~Huh.
  \newblock {\em Lorentzian polynomials.}.
  \newblock {\em Ann. of Math. (2)}, 192(3): 821--891, 2020.
  
  \bibitem{BLP}
  P.~Brändén,~J.~Leake,~I.~Pak.
  \newblock{\em Lower Bounds for contingency tables via Lorentzian Polynomials} 
  \newblock{\em Israel J. Math.} 253(1): 43--90, 2023.
  
  \bibitem{DGH}
  M.~Dyer,~P.~Gritzmann,~A.~Hufnagel.
  \newblock{\em On the Complexity of Computing Mixed Volumes}
  \newblock{\em SIAM Journal on Computing} 27(2): 356--400, 1998.

  \bibitem{F}
  W.~Fenchel.
  \newblock {\em In\'egalit\'es quadratiques entre les volumes mixtes des corps convexes}.
  \newblock {\em C. R. Acad. Sci., Paris}, 203: 647--650, 1936.

  \bibitem{Gur}
  L.~Gurvits.
  \newblock {\em On multivariates Newton-like nequalities}.
  \newblock {\em Adv. in Comb. Math.}, 61--78, 2009.
  
  \bibitem{Gur2}
  L.~Gurvits.
  \newblock{\em A polynomial time algorithm to approximate the mixed volume within a simply exponential factor}
  \newblock{\em Discrete Comput. Geom.}, 41: 533--555, 2009.
  
  \bibitem{H}
  R.~Heine.
  \newblock{\em Der {W}ertvorrat der gemischten {I}nhalte von zwei, drei und vier ebenen {E}ibereichen}
  \newblock {\em Math. Ann.}, 115(1): 115--129, 1938.
  
  \bibitem{Hu}
  J.~Huh.
  \newblock{\em Combinatorics and Hodge Theory}
  \newblock{\em Proceedings of the International Congress of Mathematicians 1}, 2022.
  
  \bibitem{JX}
  J.~Xiao.
  \newblock{\em Bézout type inequality in convex geometry}   
  \newblock{\em Int. Math. Res. Not. IMRN}, 16: 4950--4965, 2019.
  
  \bibitem{LX}
  B.~Lehmann, J.~Xiao.
  \newblock{\em Correspondences between convex geometry and complex geometry}
  \newblock{\em \'{E}pijournal de G\'{e}om\'{e}trie Alg\'{e}brique. EPIGA}, 1: Art. 6, 29, 2017.
  
  \bibitem{M}
  H.~Minkowski.
  \newblock{\em Volumen und Oberfläche}
  \newblock{\em Math. Ann.}, 57(4): 447--495, 1903.
  
  \bibitem{Schn}
  R.~Schneider.
  \newblock {\em Convex bodies: the Brunn-Minkowski theory,}.
  \newblock Cambridge University Press, Cambridge, 2014.
  
  \bibitem{Sh}
  G.~Shephard.
  \newblock{\em Inequalities between mixed volumes of convex sets.} 
  \newblock {\em Mathematika}, 7: 125--138, 1960.
  
  \end{thebibliography}
\end{document}